\UseRawInputEncoding
\documentclass[12pt]{amsart}
\usepackage{}

\usepackage{amsmath}
\usepackage{amsfonts}
\usepackage{amssymb}
\usepackage[all]{xy}           %xypic macro for latex2.09

\usepackage{bbding}
\usepackage{txfonts}
\usepackage{amscd}

\usepackage[shortlabels]{enumitem}
\usepackage{ifpdf}
\ifpdf
  \usepackage[colorlinks,final,backref=page,hyperindex]{hyperref}
\else
  \usepackage[colorlinks,final,backref=page,hyperindex,hypertex]{hyperref}
\fi
\usepackage{tikz}
\usepackage[active]{srcltx}

\makeatletter

\newtheorem{thm}{Theorem}[section]
\newtheorem{lem}[thm]{Lemma}
\newtheorem{cor}[thm]{Corollary}
\newtheorem{pro}[thm]{Proposition}
\newtheorem{ex}[thm]{Example}
\newtheorem{rmk}[thm]{Remark}
\newtheorem{defi}[thm]{Definition}

\newcommand {\emptycomment}[1]{}

\newcommand{\lon }{\,\rightarrow\,}
\newcommand{\be }{\begin{equation}}
\newcommand{\ee }{\end{equation}}

\newcommand{\g}{\mathfrak g}
\newcommand{\h}{\mathfrak h}

\newcommand{\huaB}{\mathcal{B}}%{{\mathcal{E}}}%{\mathcal{B}}

\newcommand{\huaL}{\mathcal{L}}

\newcommand{\huaP}{\mathcal{P}}
\newcommand{\huaH}{\mathcal{H}}

\newcommand{\huaZ}{\mathcal{Z}}

\newcommand{\frkC}{\mathfrak C}

\newcommand{\dM}{\mathrm{d}}

\newcommand{\Id}{{\rm{Id}}}

\newcommand{\br}[1]{   [ \cdot,    \cdot  ]   }

\newcommand{\Hom}{\mathrm{Hom}}

\newcommand{\Der}{\mathrm{Der}}

\newcommand{\Ob}{\mathsf{Ob_{T}}}

\newcommand{\gl}{\mathfrak {gl}}

\newcommand{\ad}{\mathrm{ad}}

\newcommand{\V}{\mathbb{V}}

\newcommand{\Le}{\mathsf{3Leib}}

\begin{document}

\title[Deformations and cohomologies of embedding tensors on 3-Lie algebras]{Deformations and cohomologies of embedding tensors on 3-Lie algebras}

\author{Meiyan Hu}
\address{Department of Mathematics, Jilin University, Changchun 130012, Jilin, China}
\email{hmy21@mails.jlu.edu.cn}

\author{Shuai Hou}
\address{Department of Mathematics, Jilin University, Changchun 130012, Jilin, China}
\email{houshuai19@mails.jlu.edu.cn}

\author{Lina Song}
\address{Department of Mathematics, Jilin University, Changchun 130012, Jilin, China}
\email{songln@jlu.edu.cn}

\author{Yanqiu Zhou}
\address{School of Science, Guangxi University of Science and Technology, Liuzhou 545006, China}
\email{zhouyanqiunihao@163.com}

%\date{\today}

\begin{abstract}
In this paper, first we introduce the notion of an embedding tensor on a $3$-Lie algebra, which naturally induces a $3$-Leibniz algebra.  Using the derived bracket, we construct a Lie $3$-algebra, whose Maurer-Cartan elements are
embedding tensors. Consequently, we obtain the  $L_{\infty}$-algebra that governs deformations of embedding tensors. We define the cohomology theory for embedding tensors on $3$-Lie algebras. As applications,
we show that if two formal deformations of an embedding tensor on a $3$-Lie algebra are equivalent, then
their infinitesimals are in the same cohomology class in the second cohomology group. Moreover, an order $n$ deformation of an embedding tensor  is extendable if and only if the obstruction class, which is in the third cohomology group, is trivial.
\end{abstract}

\renewcommand{\thefootnote}{}
\footnotetext{2020 Mathematics Subject Classification. 17A42, 17B56, 17B38}

\keywords{embedding tensor, $3$-Lie algebra,  representation, cohomology, deformation}

\maketitle

\tableofcontents

\allowdisplaybreaks

%\end{document}

\section{Introduction}
The notion of embedding tensors  have been widely studied and applied in the fields of mathematics and physics.
The emergence of embedding tensors can be traced back to the study of the gauged supergravity theory \cite{NS}. In \cite{Bergshoeff}, the authors used embedding tensors to studied the N =8 supersymmetric gauge theories and Bagger-Lambert theory of multiple M2-branes.
Kotov and Strobl used embedding tensors to construct tensor hierarchies, which shows us the possible mathematical properties of embedding tensors from the physics point of view \cite{KS}.
%Kelly and Miller et al. studied the average operators on Banach algebras by using the analysis theory \cite{Kelly,Miller,Rota}.%
In mathematics, embedding tensors are called averaging operators.  Averaging operators on various types of algebras such as associative algebras and Lie algebras
were studied in \cite{Aguiar}. In particular, an averaging operator on a Lie algebra can give rise to a Leibniz
algebra structure. The controlling algebras and   cohomology theories of
embedding tensors were given in   \cite{Sheng}.
In addition, averaging operators are closely related to the operad theory
\cite{PBGN,PL} and  double algebras theory. See \cite{zhangguo} for more details.

$3$-Lie algebras and more generally, $n$-Lie
algebras were introduced by Filippov in \cite{Filippov1}, which can be regarded as generalizations of
Lie algebras to higher arities.
Nambu proved that the $n$-Lie algebra is the algebraic structure corresponding to Nambu mechanics \cite{Nambu}.   $3$-Lie algebras have various applications in mathematical physics, e.g. metric $3$-Lie algebras can be used to construct the basic model of Bagger-Lambert-Gustavsson theory \cite{dMFM}, and a special class of $3$-Lie algebras constructed from the octonions can be used to study the Chern-Simons-matter theory and gives a possible candidate for the theory on multiple M2-branes with gauge symmetry \cite{Masahito}. The deformation problem of $n$-Lie algebras and $3$-Lie algebras were studied respectively in \cite{deformation,Makhlouf}. See the review \cite{review} for more details.
Recently, the classical Yang-Baxter equation, (relative) Rota-Baxter operators, product structures and complex structures on 3-Lie algebras
were studied in \cite{BGL,BGLW,ShengTang,TR1}.
Casas, Loday and Pirashvili introduced the $n$-ary version of Leibniz algebras \cite{Casas}, which is called   $n$-Leibniz algebras, whose skew-symmetric counterparts are $n$-Lie algebras.
In \cite{Rotkiewicz}, the author introduced a graded Lie brackets on the space of cochains of $n$-Leibniz algebras and describe an $n$-Leibniz  algebra  structure as a canonical structure.

Due to the importance of embedding tensors on Lie algebras and $3$-Lie algebras, the purpose of this paper is to introduce the concept of embedding tensors on  $3$-Lie algebras, and study the corresponding deformation and cohomology theory.  A linear map $T:V\rightarrow\g$ is called an {\bf embedding tensor} on a $3$-Lie algebra $(\g,[\cdot,\cdot,\cdot]_\g)$ with respect to a representation $(V;\rho)$ if $T$ satisfies
$$
 [Tu,Tv,Tw]_{\g}=T\Big(\rho(Tu,Tv)w\Big), \quad \forall u, v, w \in V.
$$
We show that an embedding tensor $T:V \rightarrow\g$ on a $3$-Lie algebra $\g$ with
respect to a representation on $V$ naturally induces a $3$-Leibniz algebra structure on
$V$ with a representation on $\g$. The corresponding cohomology of the $3$-Leibniz algebra $V$ with coefficients in $\g$ is taken to be the
cohomology of the embedding tensor $T$. Moreover, we use Voronov's higher derived brackets to
construct a Lie $3$-algebra whose Maurer-Cartan elements are embedding tensors on 3-Lie algebras. We also obtain the twisted $L_{\infty}$-algebra that governs deformations of an embedding tensor.
Finally, we use the cohomology theory of embedding tensors to study formal deformations and the extendability of an order $n$ deformation to an order $n+1$ deformation of embedding tensors on $3$-Lie algebras. Since both embedding tensors and $3$-Lie algebras have fruitful applications in mathematical physics, embedding tensors on $3$-Lie algebras have also possible applications in mathematical physics, which will be studied in the future.

 The paper is organized as follows. In Section \ref{sec:rc}, we introduce the notion of embedding tensors on $3$-Lie algebras, which naturally induce 3-Leibniz algebra structures. We show that embedding tensors can be characterized by graphs
of the hemisemidirect product $3$-Leibniz algebra.  In Section \ref{sec:ET}, we
construct an $L_{\infty}$-algebra whose Maurer-Cartan elements are embedding tensors on $3$-Lie algebras. We also construct the $L_{\infty}$-algebra that governs  deformations of embedding tensors. In Section \ref{cohomology}, we
introduce a cohomology theory of embedding tensors on 3-Lie algebras. In Section \ref{defor}, we study formal
deformations  and extendability of  order $n$ deformations of an embedding tensor  using the established cohomology theory.

\vspace{2mm}

In this paper, we work over an algebraically closed filed $\mathbb K$ of characteristic $0$.% and all the vector spaces are over $\mathbb K$ and finite-dimensional.

\vspace{2mm}

{\bf Acknowledgements. } This research is supported by NSFC (12001226).

\section{Embedding tensors on $3$-Lie algebras}\label{sec:rc}
In this section,  we introduce the notion of embedding tensors on $3$-Lie algebras. We show that a linear map $T:V\rightarrow\g$ is an embedding tensor if and only if the graph of $T$ is a subalgebra of the hemisemidirect product $3$-Leibniz algebra $\g\ltimes_{\rho} V$. Consequently, an embedding tensor $T$ induces a $3$-Leibniz algebra structure on $V$. Finally, we provide some examples of embedding tensors on $3$-Lie algebras.
\begin{defi}\label{def:=3-lie}\cite{Filippov1}
A {\bf $3$-Lie algebra} is a vector space $\g$ together with a skew-symmetric linear map $[\cdot,\cdot,\cdot]_{\g}:\otimes^{3}\g\rightarrow \g$ such that the following {\bf Fundamental Identity} holds:
\begin{equation}\label{eq:=3-lie}
[x_{1},x_{2},[x_{3},x_{4},x_{5}]_{\g}]_{\g}-[[x_{1},x_{2},x_{3}]_{\g},x_{4},x_{5}]_{\g}-[x_{3},[x_{1},x_{2},x_{4}]_{\g},x_{5}]_{\g}-[x_{3},x_{4},[x_{1},x_{2},x_{5}]_{\g}]_{\g}=0,
\end{equation}
for all $x_i \in \g,1\leq i\leq 5.$
\end{defi}
\begin{defi}\label{def:=3-lie-representation}\cite{Kasymov1}
A {\bf representation} of a $3$-Lie algebra $(\g,[\cdot,\cdot,\cdot]_{\g})$ on a vector space $V$ is a linear map $\rho:\wedge^{2}\g \rightarrow \gl(V)$, such that for all $x_{1}, x_{2}, x_{3}, x_{4}\in \g,$ the following equalities hold:
\begin{eqnarray}
\label{rep1:=3-lie}\rho(x_{1},x_{2})\rho(x_{3},x_{4})&=&\rho([x_{1},x_{2},x_{3}]_{\g},x_{4})+\rho(x_{3},[x_{1},x_{2},x_{4}]_{\g})+\rho(x_{3},x_{4})\rho(x_{1},x_{2}),\\
\label{rep2:=3-lie}\rho(x_{1},[x_{2},x_{3},x_{4}]_{\g})&=&\rho(x_{3},x_{4})\rho(x_{1},x_{2})-\rho(x_{2},x_{4})\rho(x_{1},x_{3})+\rho(x_{2},x_{3})\rho(x_{1},x_{4}).
\end{eqnarray}
\end{defi}

Let $(\g,[\cdot,\cdot,\cdot]_{\g})$ be a $3$-Lie algebra. The linear map $\ad: \wedge^2\g\rightarrow \gl(\g)$ defined by
\begin{equation}\label{eq:=ad-rep}
\ad_{x,y}z:=[x,y,z]_{\g},\quad \forall x,y,z \in \g,
\end{equation}
is a representation of $(\g,[\cdot,\cdot,\cdot]_{\g})$ on itself, which is called the {\bf adjoint representation} of $\g$.

\begin{defi}\cite{Casas}
A {\bf $3$-Leibniz algebra} is a vector space $\mathcal{L}$ equipped with a linear map $[\cdot,\cdot,\cdot]_{\mathcal{L}}:\mathcal{L}\otimes\mathcal{L}\otimes\mathcal{L}\rightarrow \mathcal{L}$
 such that
 \begin{equation}\label{3-leib1}
{[x_{1},x_{2},[y_{1},y_{2},y_{3}]_{\mathcal{L}}]}_{\mathcal{L}}=[[x_{1},x_{2},y_{1}]_{\mathcal{L}},y_{2},y_{3}]_{\mathcal{L}}+[y_{1},[x_{1},x_{2},y_{2}]_{\mathcal{L}},y_{3}]_{\mathcal{L}}+[y_{1},y_{2},[x_{1},x_{2},y_{3}]_{\mathcal{L}}]_{\mathcal{L}}, \end{equation}
for all $ x_{1},x_{2},y_{1},y_{2},y_{3} \in \mathcal{L}.$

\end{defi}

Let $(\g,[\cdot,\cdot,\cdot]_{\g})$ be a $3$-Lie algebra and $(V;\rho)$ be a representation of $\g.$ On the direct sum vector space $\g\oplus V$, define a trilinear bracket operation $[\cdot,\cdot,\cdot]_{\rho}$ by
\begin{eqnarray}
  [x+u,y+v,z+w]_{\rho}=[x,y,z]_{\g}+\rho(x,y)w,\quad \forall x,y,z\in \g,u,v,w\in V.
\end{eqnarray}
\begin{pro}\label{hemi-direct}
  With the above notations, $(\g\oplus V,[\cdot,\cdot,\cdot]_{\rho})$ is a $3$-Leibniz algebra, which is called the {\bf hemisemidirect product $3$-Leibniz algebra}, and denoted by $\g\ltimes_{\rho} V$.
\end{pro}
\begin{proof}
For all $x_i \in \g, u_i \in V, 1 \leq i\leq 5,$ by \eqref{eq:=3-lie}-\eqref{rep1:=3-lie}, we have
\begin{eqnarray*}
&&[x_1+u_1,x_2+u_2,[x_3+u_3,x_4+u_4,x_5+u_5]_{\rho}]_{\rho}-[[x_1+u_1,x_2+u_2,x_3+u_3]_{\rho},x_4+u_4,x_5+u_5]_{\rho}\\
&&-[x_3+u_3,[x_1+u_1,x_2+u_2,x_4+u_4]_{\rho},x_5+u_5]_{\rho}\\
&&-[x_3+u_3,x_4+u_4,[x_1+u_1,x_2+u_2,x_5+u_5]_{\rho}]_{\rho}\\
&=&[x_1,x_2,[x_3,x_4,x_5]_{\g}]_{\g}+\rho(x_1,x_2)\rho(x_3,x_4)u_5-[[x_1,x_2,x_3]_{\g},x_4,x_5]_{\g}-\rho([x_1,x_2,x_3]_{\g},x_4)u_5\\
&&-[x_3,[x_1,x_2,x_4]_{\g},x_5]_{\g}-\rho(x_3,[x_1,x_2,x_4]_{\g})u_5-[x_3,x_4,[x_1,x_2,x_5]_{\g}]_{\g}-\rho(x_3,x_4)\rho(x_1,x_2)u_5\\
&=&0.
\end{eqnarray*}
Therefore, $(\g\oplus V,[\cdot,\cdot,\cdot]_{\rho})$ is a $3$-Leibniz algebra.
\end{proof}

\begin{defi}
Let $(V;\rho)$ be a representation of a $3$-Lie algebra $(\g,[\cdot,\cdot,\cdot]_\g)$. A linear map $T:V\rightarrow\g$ is called an {\bf embedding tensor} on the $3$-Lie algebra $(\g,[\cdot,\cdot,\cdot]_\g)$ with respect to the representation $(V;\rho)$ if $T$ satisfies
\begin{equation}\label{ET1}
 [Tu,Tv,Tw]_{\g}=T\Big(\rho(Tu,Tv)w\Big), \quad \forall u, v, w \in V.
\end{equation}
\end{defi}

The identity \eqref{ET1} can be characterized by the graph of $T$ being a subalgebra.
\begin{thm}\label{semi-direct}
A linear map $T:V\rightarrow\g$ is an embedding tensor on a $3$-Lie algebra $\g$ with respect to the representation $(V;\rho)$ if and only if the graph $Gr(T)=\{Tu+u|u\in V\}$ is a $3$-Leibniz subalgebra of the hemisemidirect product  $3$-Leibniz algebra $\g\ltimes_{\rho} V.$
\end{thm}

\begin{proof}
Let $T$ be a linear map. For all $u,v,w \in V$, we have
\begin{eqnarray*}
[Tu+u,Tv+v,Tw+w]_{\rho}=[Tu, Tv, Tw]_{\g}+\rho(Tu,Tv)w.
\end{eqnarray*}
Therefore, the graph $Gr(T)=\{Tu+u|u\in V\}$ is a subalgebra of the hemisemidirect product  $3$-Leibniz algebra $\g\ltimes_{\rho} V$ if and only if $T$ satisfies
$[Tu,Tv,Tw]_{\g}=T\big(\rho(Tu,Tv)w\big)$, which implies that $T$ is an embedding tensor on a $3$-Lie algebra $\g$ with respect to the representation $(V;\rho)$.
\end{proof}

The algebraic structure underlying an embedding tensor on the $3$-Lie algebra $(\g,[\cdot,\cdot,\cdot]_\g)$ is a $3$-Leibniz algebra. We have the following proposition.
\begin{pro}\label{ET}
Let $T:V\rightarrow\g$ be an embedding tensor on the $3$-Lie algebra  $(\g,[\cdot,\cdot,\cdot]_\g)$  with respect to the representation $(V;\rho)$. Define a linear map $[\cdot,\cdot,\cdot]_{T}:\otimes^{3}V\rightarrow V$ by
\begin{equation}\label{ET2}
[u,v,w]_{T}=\rho(Tu,Tv)w,\quad \forall u,v,w \in V.
\end{equation}
Then $(V,[\cdot,\cdot,\cdot]_{T})$ is a $3$-Leibniz algebra. Moreover, $T$ is a homomorphism from the $3$-Leibniz algebra $(V,[\cdot,\cdot,\cdot]_{T})$ to the $3$-Lie algebra $(\g,[\cdot,\cdot,\cdot]_{\g})$.
\end{pro}
\begin{proof}
For any $u_{1},u_{2},u_{3},u_{4},u_{5}$, by \eqref{rep1:=3-lie}, \eqref{ET1} and \eqref{ET2}, we have
\begin{eqnarray*}
&&[[u_{1},u_{2},u_{3}]_{T},u_{4},u_{5}]_{T}+[u_{3},[u_{1},u_{2},u_{4}]_{T},u_{5}]_{T}+[u_{3},u_{4},[u_{1},u_{2},u_{5}]_{T}]_{T}\\
&=&[\rho(Tu_{1},Tu_{2})u_{3},u_{4},u_{5}]_{T}+[u_{3},\rho(Tu_{1},Tu_{2})u_{4},u_{5}]_{T}+[u_{3},u_{4},\rho(Tu_{1},Tu_{2})u_{5}]_{T}\\
&=&\rho(T(\rho(Tu_{1},Tu_{2})u_{3}),Tu_{4})u_{5}+\rho(Tu_{3},T(\rho(Tu_{1},Tu_{2})u_{4}))u_{5}+\rho(Tu_{3},Tu_{4})\rho(Tu_{1},Tu_{2})u_{5}\\
&=&\rho\big([Tu_{1},Tu_{2},T u_{3}]_{\g},Tu_{4}\big)u_{5}+\rho\big(Tu_{3},[Tu_{1},Tu_{2},T u_{4}]_{\g}\big)u_{5}+\rho\big(Tu_{3},Tu_{4})\rho(Tu_{1},Tu_{2}\big)u_{5}\\
&=&\rho(Tu_{1},Tu_{2})\rho(Tu_{3},Tu_{4})u_{5}\\
&=&[u_{1},u_{2},\rho(Tu_{3},Tu_{4})u_{5}]_{T}\\
&=&[u_{1},u_{2},[u_{3},u_{4},u_{5}]_{T}]_{T}.
\end{eqnarray*}
Therefore, $(V,[\cdot,\cdot,\cdot]_{T})$ is a $3$-Leibniz algebra. By \eqref{ET1}, $T$ is a homomorphism from the $3$-Leibniz algebra $(V,[\cdot,\cdot,\cdot]_{T})$ to the $3$-Lie algebra $(\g,[\cdot,\cdot,\cdot]_{\g})$.
\end{proof}

\begin{defi}\label{homoET}
Let $T$ and $T^{\prime}$ be two embedding tensors on the $3$-Lie algebra $(\g,[\cdot,\cdot,\cdot]_{\g})$ with respect to the representation $(V;\rho)$. A {\bf homomorphism} from
$T^{\prime}$ to $T$ consists of a $3$-Lie algebra homomorphism $\phi_{\g}:\g\rightarrow \g$ and a linear map $\phi_{V}:V\rightarrow V$ such that
\begin{eqnarray}
\label{homomorphism-1-3Lei}T\circ \phi_{V}&=&\phi_{\g}\circ T^{\prime},\\
\label{homomorphism-2-3Lei}\phi_{V}(\rho(x,y)u)&=&\rho(\phi_{\g}(x),\phi_{\g}(y))(\phi_{V}(u)), \quad \forall x,y \in \g, u\in V.
\end{eqnarray}
In particular, if both $\phi_{\g}$ and $\phi_{V}$ are invertible, $(\phi_{\g},\phi_{V})$ is called an {\bf isomorphism} from $T'$ to $T$.
\end{defi}

\begin{pro}
Let $T$ and $T^{\prime}$ be two embedding tensors on the $3$-Lie algebra $(\g,[\cdot,\cdot,\cdot]_{\g})$ with respect to the representation $(V;\rho)$ and $(\phi_{\g},\phi_{V})$ be a homomorphism from $T^{\prime}$ to $T$. Then $\phi_{V}$ is a homomorphism of $3$-Leibniz algebras from $(V,[\cdot,\cdot,\cdot]_{T^{\prime}})$ to $(V,[\cdot,\cdot,\cdot]_{T})$.
\end{pro}
\begin{proof}
For all $u,v,w$, by \eqref{homomorphism-1-3Lei}-\eqref{homomorphism-2-3Lei}, we have
\begin{eqnarray*}
\phi_{V}([u,v,w]_{T^{\prime}})&=&\phi_{V}(\rho(T^{\prime}u,T^{\prime}v)w)\\
&=&\rho\big(\phi_{\g}(T^{\prime}u),\phi_{\g}(T^{\prime}v)\big)(\phi_{V}(w))\\
&=&\rho\big(T(\phi_{V}(u)),T(\phi_{V}(v))\big)(\phi_{V}(w))\\
&=&[\phi_{V}(u),\phi_{V}(v),\phi_{V}(w)]_{T}.
\end{eqnarray*}
Therefore, $\phi_{V}$ is a homomorphism of $3$-Leibniz algebras from $(V,[\cdot,\cdot,\cdot]_{T^{\prime}})$ to $(V,[\cdot,\cdot,\cdot]_{T})$.
\end{proof}

At the end of this section, we present some examples of embedding tensors on $3$-Lie algebras.
\begin{ex}
Recall that a derivation $D$ on a $3$-Lie algebra $(\g,[\cdot,\cdot,\cdot]_{\g})$ is a linear map $D:\g\rightarrow \g$ satisfying
\begin{equation*}
  D[x_{1},x_{2},x_{3}]_{\g}=[D(x_{1}),x_{2},x_{3}]_{\g}+[x_{1},D(x_{2}),x_{3}]_{\g}+[x_{1},x_{2},D(x_{3})]_{\g},\quad \forall x_{1},x_{2},x_{3} \in \g.
\end{equation*}
If $D \circ D=0,$ we have
\begin{eqnarray*}
D[D(x_{1}),D(x_{2}),x_{3}]_{\g}&=&[D^{2}(x_{1}),D(x_{2}),x_{3}]_{\g}+[D(x_{1}),D^{2}(x_{2}),x_{3}]_{\g}+[D(x_{1}),D(x_{2}),D(x_{3})]_{\g}\\
&=&[D(x_{1}),D(x_{2}),D(x_{3})]_{\g}.
\end{eqnarray*}
Then $D$ is an embedding tensor on the $3$-Lie algebra $(\g,[\cdot,\cdot,\cdot]_{\g})$ with respect to the adjoint representation $(\g;\ad)$.
\end{ex}
The notion of a strict $3$-Lie $2$-algebra was introduced in \cite{ZLS}.
\begin{ex}
A strict $3$-Lie $2$-algebra is a $2$-term graded vector spaces $V = V_{1}\oplus V_{0}$ equipped with a linear map $\dM :V_{1}\rightarrow V_{0}$ and
skew-symmetric trilinear maps $[\cdot,\cdot,\cdot]: V_{i}\wedge V_{j}\wedge V_{k}\rightarrow V_{i+j+k}$, where $0 \leq i+j+k \leq 1$, such that for all $x, y, x_{i} \in V_{0}, i=1,\cdots, 5$ and $f, g \in V_{1}$, the following equalities are satisfied{\rm:}
\begin{itemize}
\item[$\rm(i)$] $\dM[x,y,f]=[x,y,\dM f],\quad[\dM f,g,x]=[f,\dM g,x],$
\item[$\rm(ii)$]$[x_1,x_2,[x_3,x_4,x_5]]=[[x_1,x_2,x_3],x_4,x_5]+[x_3,[x_1,x_2,x_4],x_5]+[x_3,x_4,[x_1,x_2,x_5]],$
\item[$\rm(iii)$]$[f,x_2,[x_3,x_4,x_5]]~=[[f,x_2,x_3],x_4,x_5]+[x_3,[f,x_2,x_4],x_5]+[x_3,x_4,[f,x_2,x_5]],$
\item[$\rm(iv)$]$[x_1,x_2,[f,x_4,x_5]]~=[[x_1,x_2,f],x_4,x_5]+[f,[x_1,x_2,x_4],x_5]+[f,x_4,[x_1,x_2,x_5]].$
\end{itemize}
Define a linear map $\rho:\wedge^{2}V_0\rightarrow \gl(V_1)$ by
\begin{equation*}
  \rho(x,y)(f)=[x,y,f].
\end{equation*}
Then by $\rm(iii)$ and $\rm(iv)$, $(V_1,\rho)$ is a representation of the $3$-Lie algebra $(V_0,[\cdot,\cdot,\cdot])$. By $\rm(i)$,
$\dM$ is an embedding tensor on the $3$-Lie algebra $(V_0,[\cdot,\cdot,\cdot])$ with respect to the representation $(V_1;\rho)$.
\end{ex}
There is a one-to-one correspondence between strict 3-Lie 2-algebras and crossed modules of 3-Lie algebras \cite{ZLS}, then we have the following example.

\begin{ex}
A crossed module of $3$-Lie algebras is a quadruple $((\g, [\cdot,\cdot,\cdot]_{\g}),(\h, [\cdot,\cdot,\cdot]_{\h}), \mu,\alpha)$, where $(\g, [\cdot,\cdot,\cdot]_{\g})$ and $(\h, [\cdot,\cdot,\cdot]_{\h})$ are $3$-Lie algebras, $\mu:\g\rightarrow \h$ is a homomorphism of $3$-Lie
algebras, and $\alpha: \wedge^{2}\h\rightarrow \Der(\g)$ is a representation, such that for all $x, y \in \h, f, g, h \in \g$, the
following equalities hold{\rm :}
\begin{eqnarray*}
\mu(\alpha(x, y)(f)) &=& [x, y, \mu(f)]_{\h}, \\
\alpha(\mu(f), \mu(g))(h) &=& [f, g, h]_{\g},\\
\alpha(x, \mu(f))(g) &=& -\alpha(x, \mu(g))(f).
\end{eqnarray*}
Then $\mu$ is an embedding tensor on the $3$-Lie algebra $(\h, [\cdot,\cdot,\cdot]_{\h})$ with respect to the representation $(\g;\alpha)$.
\end{ex}

\begin{ex}
Let $(V ; \rho)$ be a representation of a $3$-Lie algebra $(\g, [\cdot,\cdot,\cdot]_{\g})$. If a linear map $T: V \rightarrow \g$ satisfies
\begin{equation*}
  T(\rho(x,Tu)v)=[x,Tu,Tv]_{\g}, \quad \forall x \in \g, u,v \in V.
\end{equation*}
Then T is an embedding tensor on the $3$ Lie algebra $(\g,[\cdot,\cdot,\cdot]_{\g})$ with respect to the representation $(V;\rho)$.
\end{ex}

\section{Maurer-Cartan characterization of embedding tensors on $3$-Lie algebras}\label{sec:ET}
In this section, we construct a Lie $3$-algebra whose Maurer-Cartan elements are embedding tensors on $3$-Lie algebras. Moreover, we obtain the twisted $L_{\infty}$-algebra that controls deformations of embedding tensors on $3$-Lie algebras.

The notion of $L_{\infty}$-algebras (also called strongly homotopy Lie algebras) were introduced by Schlessinger and Stasheff \cite{Lada,SS85}. $L_{\infty}$-algebras can be understood as  a generalization of differential graded Lie algebras in which the Jacobi rule is only satisfied up
to a hierarchy of higher homotopies.

A permutation $\sigma\in\mathbb S_n$ is called an $(i,n-i)$-{\bf shuffle} if $\sigma(1)<\cdots <\sigma(i)$ and $\sigma(i+1)<\cdots <\sigma(n)$. If $i=0$ or $n$ we assume $\sigma=\Id$. The set of all $(i,n-i)$-shuffles will be denoted by $\mathbb S_{(i,n-i)}$.
\begin{defi}{\rm (\cite{stasheff:shla})}
An {\em  $L_\infty$-algebra} is a $\mathbb Z$-graded vector space $\g=\oplus_{k\in\mathbb Z}\g^k$ equipped with a collection $(k\ge 1)$ of linear maps $l_k:\otimes^k\g\rightarrow\g$ of degree $1$ with the property that, for any homogeneous elements $x_1,\cdots,x_n\in \g$, we have
\begin{itemize}
\item[\rm(i)]
{\em (graded symmetry)} for every $\sigma\in\mathbb S_{n}$,
\begin{eqnarray*}
l_n(x_{\sigma(1)},\cdots,x_{\sigma(n-1)},x_{\sigma(n)})=\varepsilon(\sigma)l_n(x_1,\cdots,x_{n-1},x_n).
\end{eqnarray*}
\item[\rm(ii)] {\em (generalized Jacobi identity)} for all $n\ge 1$,
\begin{eqnarray*}\label{sh-Lie}
\sum_{i=1}^{n}\sum_{\sigma\in \mathbb S_{(i,n-i)} }\varepsilon(\sigma)l_{n-i+1}(l_i(x_{\sigma(1)},\cdots,x_{\sigma(i)}),x_{\sigma(i+1)},\cdots,x_{\sigma(n)})=0.
\end{eqnarray*}
\end{itemize}
\end{defi}

A Lie $k$-algebra is a special $L_\infty$-algebra which was introduced by Hanlon and Wachs in \cite{Hanlon-Wachs}, which only equipped with the $k$-ary bracket, and other operations are trivial.

\begin{defi}\label{symmetric-lie}
A {\bf Lie $3$-algebra} is a $\mathbb Z$-graded vector space $\g=\oplus_{k\in\mathbb Z}\g^k$ equipped with a trilinear bracket $
\{\cdot,\cdot,\cdot\}_\g:\g\otimes \g\otimes\g\lon \g$  of degree $1$, satisfying
\begin{itemize}\item[\rm(i)]
{\em (graded symmetry)} for all homogeneous elements  $x_1,x_2,x_3\in\g$,
\begin{eqnarray}\label{Lie-3-alg-1}
\{x_1,x_2,x_3\}_\g=(-1)^{x_1x_2}\{x_2,x_1,x_3\}_\g=(-1)^{x_2x_3}\{x_1,x_3,x_2\}_\g.
\end{eqnarray}
\item[\rm(ii)] {\em (generalized Jacobi identity)} for all homogeneous elements $ x_{i}\in \g, 1\leq i\leq 5$,
\begin{eqnarray}\label{Lie-3-alg-2}
\sum_{\sigma\in \mathbb S_{5} }\varepsilon(\sigma)\{\{x_{\sigma(1)},x_{\sigma(2)},x_{\sigma(3)}\}_\g,x_{\sigma(4)},x_{\sigma(5)}\}_\g=0.
\end{eqnarray}
\end{itemize}
\end{defi}

\begin{defi}
\begin{itemize}
\item[{\rm(i)}] A {\bf  Maurer-Cartan element} of an $L_\infty$-algebra $(\g,\{l_i\}_{i=1}^{+\infty})$ is an element $\alpha\in \g^0$ satisfying the Maurer-Cartan equation
\begin{eqnarray}\label{MC-equationL}
\sum_{n=1}^{+\infty} \frac{1}{n!}l_n(\alpha,\cdots,\alpha)=0.
\end{eqnarray}

  \item[{\rm(ii)}]
  A {\bf  Maurer-Cartan element} of a Lie $3$-algebra $(\g,\{\cdot,\cdot,\cdot\}_\g)$ is an element $\alpha\in \g^0$ satisfying the Maurer-Cartan equation
\begin{eqnarray}\label{MC-equation}
\frac{1}{3!}\{\alpha,\alpha,\alpha\}_\g=0.
\end{eqnarray}
\end{itemize}
\end{defi}

Let $\alpha$ be a Maurer-Cartan element of a Lie $3$-algebra $(\g,\{\cdot,\cdot,\cdot\}_\g)$. For all $k\geq1$ and $x_1,\cdots,x_k\in \g,$
define a series of linear maps $l_k^\alpha:\otimes^k\g\lon\g$ of degree $1$ by
\begin{eqnarray}
l_k^{\alpha}(x_1,\cdots,x_k)&=&\sum\limits_{n=0}^{+\infty}\frac{1}{n!}l_{k+n}\{\underbrace{\alpha,\cdots,\alpha}_{n},x_1,\cdots,x_k\}_{\g}.
\end{eqnarray}
\begin{lem}{\rm (\cite{Getzler})}\label{Getzler-th}
With the above notations, $(\g,l_1^{\alpha},l_2^{\alpha},l_3^{\alpha})$ is an $L_\infty$-algebra, obtained from the Lie $3$-algebra $(\g,\{\cdot,\cdot,\cdot\}_\g)$ by twisting with the Maurer-Cartan element $\alpha$. Moreover, $\alpha+\alpha'$ is a Maurer-Cartan element of   $(\g,\{\cdot,\cdot,\cdot\}_\g)$ if and only if $\alpha'$ is a Maurer-Cartan element of the twisted $L_\infty$-algebra $(\g,l_1^{\alpha},l_2^{\alpha},l_3^{\alpha})$.
\end{lem}

In the sequel, we recall Th. Voronov's derived brackets theory \cite{Vo}, which is a useful tool to construct explicit $L_\infty$-algebras.

\begin{defi}{\rm (\cite{Vo})}
A {\bf $V$-data} consists of a quadruple $(L,F,\mathcal{P},\Delta)$ where
\begin{itemize}
\item[$\bullet$] $(L,[\cdot,\cdot])$ is a graded Lie algebra;
\item[$\bullet$] $F$ is an abelian graded Lie subalgebra of $(L,[\cdot,\cdot])$;
\item[$\bullet$] $\huaP:L\lon L$ is a projection, that is $\mathcal{P}\circ \mathcal{P}=\mathcal{P}$, whose image is $F$ and kernel is a  graded Lie subalgebra of $(L,[\cdot,\cdot])$;
\item[$\bullet$] $\Delta$ is an element in $  \ker(\huaP)^1$ such that $[\Delta,\Delta]=0$.
\end{itemize}
\end{defi}

\begin{thm}{\rm (\cite{Vo})}\label{thm:db}
Let $(L,F,\mathcal{P},\Delta)$ be a $V$-data. Then $(F,\{{l_k}\}_{k=1}^{+\infty})$ is an $L_\infty$-algebra, where
\begin{eqnarray}\label{V-shla}
l_k(a_1,\cdots,a_k)=\mathcal{P}\underbrace{[\cdots[[}_k\Delta,a_1],a_2],\cdots,a_k],\quad\mbox{for homogeneous}~   a_1,\cdots,a_k\in F.
\end{eqnarray}
 We call $\{{l_k}\}_{k=1}^{+\infty}$ the {\bf higher derived brackets} of the $V$-data $(L,F,\mathcal{P},\Delta)$. %Moreover, we denote this $L_\infty$-algebra by $\h_{\Delta}^{P}$.
\end{thm}

Let $\g$ be a vector space. We consider the graded vector space $$C^{\ast}(\g,\g)=\oplus_{n\geq 0}C^{n}(\g,\g)=\oplus_{n\geq 0}\Hom(\underbrace{\otimes^{2}\g\otimes\cdots\otimes^{2}\g}_{n}\otimes\g,\g).$$
\begin{thm}{\rm (\cite{Rotkiewicz})}\label{thm:graded-3-Lebiniz}
 The graded vector space $C^{\ast}(\g,\g)$ equipped with the graded bracket
 \begin{equation*}
[P,Q]_{\Le}=P\circ Q-(-1)^{pq}Q\circ P, \quad \forall P \in C^{p}(\g,\g),  Q\in C^{q}(\g,\g),
\end{equation*}
 is a graded Lie algebra, where $P\circ Q\in C^{p+q}(\g,\g)$ is defined by
\begin{eqnarray*}
&&(P\circ Q)(\mathfrak{X}_{1},\cdots,\mathfrak{X}_{p+q},x)\\
&=&\sum_{k=1}^{p}(-1)^{(k-1)q}\sum_{\sigma \in S(k-1,q)}(-1)^{\sigma}P\Big(\mathfrak{X}_{\sigma(1)},\cdots,\mathfrak{X}_{\sigma(k-1)},Q(\mathfrak{X}_{\sigma(k)},\cdots,\mathfrak{X}_{\sigma(k+q-1)},x_{k+q})\otimes y_{k+q},\\
&&\mathfrak{X}_{k+q+1},\cdots,\mathfrak{X}_{p+q},x\Big)\\
&&+\sum_{k=1}^{p}(-1)^{(k-1)q}\sum_{\sigma \in S(k-1,q)}(-1)^{\sigma}P\Big(\mathfrak{X}_{\sigma(1)},\cdots,\mathfrak{X}_{\sigma(k-1)},x_{k+q}\otimes Q(\mathfrak{X}_{\sigma(k)},\cdots,\mathfrak{X}_{\sigma(k+q-1)},y_{k+q}),\\
&&\mathfrak{X}_{k+q+1},\cdots,\mathfrak{X}_{p+q},x\Big)\\
&&+\sum_{\sigma \in S(p,q)}(-1)^{pq}(-1)^{\sigma}P\Big(\mathfrak{X}_{\sigma(1)},\cdots,\mathfrak{X}_{\sigma(p)},Q(\mathfrak{X}_{\sigma(p+1)},\cdots,\mathfrak{X}_{\sigma(p+q)},x)\Big),
\end{eqnarray*}
for all $\mathfrak{X}_{i}=x_{i}\otimes y_{i}\in \otimes^{2}\g,i=1,2,\cdots,p+q$ and $x\in\g.$
\end{thm}
\begin{rmk}\label{rmk:graded-3-Lebiniz}
For $\pi \in \Hom(\otimes^{3}\g,\g)$, we have
\begin{eqnarray*}
[\pi,\pi]_{\Le}(\mathfrak{X}_{1},\mathfrak{X}_{2},x)&=&2(\pi \circ \pi)(\mathfrak{X}_{1},\mathfrak{X}_{2},x)\\
&=&2\Big(\pi(\pi(x_{1},y_{1},x_{2}),y_{2},x)+\pi(x_{2},\pi(x_{1},y_{1},y_{2}),x)\\
&&-\pi(x_{1},y_{1},\pi(x_{2},y_{2},x))+\pi(x_{2},y_{2},\pi(x_{1},y_{1},x))\Big).
\end{eqnarray*}
\end{rmk}
Thus, $\pi$ is a $3$-Leibniz algebra structure if and only if $[\pi, \pi]_{\Le}=0$, i.e. $\pi$ is a Maurer-Cartan element of the graded Lie algebra $\big(C^{\ast}(\g, \g),[\cdot,\cdot]_{\Le}\big)$.

Let $(V ; \rho)$ be a representation of a 3-Lie algebra $(\g, [\cdot,\cdot,\cdot]_{\g})$. For convenience, we use $\mu:\wedge^{3}\g\rightarrow\g$ to indicate the 3-Lie bracket $[\cdot,\cdot,\cdot]_{\g}$. In the sequel, we use $\mu\boxplus\rho$ to denote the element in $\Hom(\otimes^{3}(\g \oplus V),\g \oplus V)$ given by
\begin{eqnarray}
(\mu \boxplus \rho)\big(x+u,y+v,z+w\big)=[x,y,z]_{\g}+\rho(x,y)w,
\end{eqnarray}
for all $x,y,z\in\g,u,v,w\in V.$ Note that the right hand side is exactly the hemisemidirect product $3$-Leibniz algebra structure given in Proposition \ref{hemi-direct}. Therefore by Remark \ref{rmk:graded-3-Lebiniz}, we have
$$[\mu\boxplus\rho,\mu\boxplus\rho]_{\Le}=0.$$

\begin{lem}\label{lem-equation-1}
Let $T:V\rightarrow\g$ be an embedding tensor on a $3$-Lie algebra $(\g,[\cdot,\cdot,\cdot]_\g)$ with respect to the representation $(V;\rho)$. For all $x,y,z\in \g, u,v,w\in V,$
we have
\begin{eqnarray*}
&&[\mu\boxplus\rho,T]_{\Le}(x+u,y+v,z+w)\\
&=&[Tu,y,z]_{\g}+[x,Tv,z]_{\g}+[x,y,Tw]_{\g}-T(\rho(x,y)w)+\rho(Tu,y)w+\rho(x,Tv)w;\\
&&[[\mu \boxplus \rho,T]_{\Le},T]_{\Le}(x+u,y+v,z+w)\\
&=&2\Big([Tu,Tv,z]_{\g}+[Tu,y,Tw]_{\g}+[x,Tv,Tw]_{\g}-T(\rho(Tu,y)w+\rho(x,Tv)w)+\rho(Tu,Tv)w\Big).
\end{eqnarray*}

\end{lem}
\begin{proof}
  It follows from straightforward computations.
\end{proof}

\begin{pro}
Let $(V ; \rho)$ be a representation of a $3$-Lie algebra $(\g, [\cdot,\cdot,\cdot]_{\g})$. Then we have a $V$-data $(L,F,\huaP,\Delta)$ as follows:
\begin{itemize}
\item[$\bullet$] the graded Lie algebra $(L,[\cdot,\cdot])$ is given by $(C^*(\g\oplus V,\g\oplus V),[\cdot,\cdot]_{\Le});$
\item[$\bullet$] the abelian graded Lie subalgebra $F$ is given by
$$
F=C^*(V,\g)=\oplus_{n\geq 0}C^{n}(V,\g)=\oplus_{n\geq 0}\Hom(\underbrace{\otimes^{2} V\otimes \cdots\otimes^{2}V}_{n}\otimes V, \g)
;$$
\item[$\bullet$] $\huaP:L\lon L$ is the projection onto the subspace $F;$
\item[$\bullet$]$\Delta=\mu\boxplus\rho.$
\end{itemize}
Consequently, we obtain a Lie $3$-algebra $(C^*(V,\g),\{\cdot,\cdot,\cdot\})$, where
the trilinear bracket operation $$\{\cdot,\cdot,\cdot\}: C^p(V,\g)\times C^q(V,\g)\times C^r(V,\g)\longrightarrow C^{p+q+r+1}(V,\g)$$  is defined by
\begin{eqnarray}\label{lie2-bracket}
\{P,Q,R\}=[[[\mu \boxplus \rho,P]_{\Le},Q]_{\Le},R]_{\Le},
\end{eqnarray}
for all $P \in C^{p}(V,\g), Q \in C^{q}(V,\g)$ and $R\in C^{r}(V,\g)$.
\end{pro}
\begin{proof}
     By Theorem \ref{thm:db}, $(F,\{l_k\}_{k=1}^{+\infty})$ is an $L_\infty$-algebra, where $l_k$ is given by
\eqref{V-shla}. For all $P\in C^p(V,\g),Q\in C^q(V,\g)$ and $R\in C^r(V,\g)$, by Lemma \ref{lem-equation-1}, we have
\begin{eqnarray*}
&&[\mu\boxplus\rho,P]_{\Le}\in\ker(\mathcal{P}),\\
&&[[\mu\boxplus\rho,P]_{\Le},Q]_{\Le}\in\ker(\mathcal{P}).
\end{eqnarray*}
Similarly, we deduce that $l_k=0$ when $k\geq4$. Namely, the graded vector space $C^*(V,\g)$ is a Lie $3$-algebra with trilinear bracket operation $\{\cdot,\cdot,\cdot\}$ and other maps are trivial.
\end{proof}

\begin{thm}\label{thm:=ET1}
Let $(V ; \rho)$ be a representation of a $3$-Lie algebra $(\g, [\cdot,\cdot,\cdot]_{\g})$. Then Maurer-Cartan elements of the Lie $3$-algebra $(C^{\ast}(V,\g),\{\cdot,\cdot,\cdot\})$ are precisely embedding tensors on the $3$-Lie algebra  $(\g,[\cdot,\cdot,\cdot]_\g)$ with respect to the representation $(V;\rho)$.
\end{thm}
\begin{proof}
For a degree zero element $T: V\rightarrow \g$, by \eqref{lie2-bracket}, we have
\begin{eqnarray*}
&&\{T,T,T\}(u,v,w)\\
&=&[[[\mu \boxplus \rho,T]_{\Le},T]_{\Le},T]_{\Le}(u,v,w)\\
&=&[[\mu \boxplus \rho,T]_{\Le},T]_{\Le}(Tu,v,w)+[[\mu \boxplus \rho,T]_{\Le},T]_{\Le}(u,Tv,w)\\
&&+[[\mu \boxplus \rho,T]_{\Le},T]_{\Le}(u,v,Tw)-T[[\mu \boxplus \rho,T]_{\Le},T]_{\Le}(u,v,w)\\
&=&6\big([Tu,Tv,Tw]_{\g}-T(\rho(Tu,Tv)w\big),
\end{eqnarray*}
which indicates that a linear map $T \in \Hom(V, \g)$ is a Maurer-Cartan element of the Lie $3$-algebra $(C^{\ast}(V,\g),\{\cdot,\cdot,\cdot\})$ if and only if $T$ is an embedding tensor on the $3$-Lie algebra $(\g,[\cdot,\cdot,\cdot]_\g)$ with respect to the representation $(V;\rho)$.
\end{proof}

\begin{pro}
Let $(V ; \rho)$ be a representation of a $3$-Lie algebra $(\g, [\cdot,\cdot,\cdot]_{\g})$. Then the graded vector space $C^{\ast}(V,\g)$ carries a twisted $L_{\infty}$-algebra structure as following:
\begin{eqnarray}
\label{eq:=L1}l_1^{T}(P)&=&\frac{1}{2}\{T,T,P\},\\
\label{eq:=L2}l_2^{T}(P,Q)&=&\{T,P,Q\},\\
\label{eq:=L3}l_3^{T}(P,Q,R)&=&\{P,Q,R\},\\
\label{eq:=L4}l^{T}_k&=&0,\quad k\geq4,
\end{eqnarray}
where $P\in C^p(V,\g),Q\in C^q(V,\g)$ and $R\in C^r(V,\g)$.
\end{pro}
\begin{proof}
 Since $T$ is a Maurer-Cartan element of the Lie $3$-algebra  $(C^*(V,\g),\{\cdot,\cdot,\cdot\})$, by Lemma~\ref{Getzler-th}, the conclusion holds.
\end{proof}

\begin{thm}\label{thm:deformation}
Let $T:V\rightarrow\g$ be an embedding tensor on a $3$-Lie algebra $(\g,[\cdot,\cdot,\cdot]_\g)$ with respect to the representation $(V;\rho)$. Then for a linear map $T':V\rightarrow \g$, $T+T'$ is an embedding tensor if and only if $T'$ is a Maurer-Cartan element of the twisted $L_\infty$-algebra $\big(C^*(V,\g),l_1^{T},l_2^{T},l_3^{T}\big)$, that is $T'$ satisfies the Maurer-Cartan equation:
$$
l_1^{T}(T')+\frac{1}{2}l_2^{T}(T',T')+\frac{1}{3!}l_3^{T}(T',T',T')=0.
$$
\end{thm}
\begin{proof}
  By Theorem \ref{thm:=ET1}, $T+T'$ is an embedding tensor if and only if
  $$\frac{1}{3!}\{T+T',T+T',T+T'\}=0.$$
By $\{T,T,T\}=0,$ the above equality is equivalent to
  $$\frac{1}{2}\{T,T,T'\}+\frac{1}{2}\{T,T',T'\}+\frac{1}{6}\{T',T',T'\}=0,$$
which implies that $l_1^{T}(T')+\frac{1}{2}l_2^{T}(T',T')+\frac{1}{3!}l_3^{T}(T',T',T')=0.$
Namely, $T'$ is a Maurer-Cartan element of the twisted $L_\infty$-algebra $\big(C^*(V,\g),l_1^{T},l_2^{T},l_3^{T}\big)$.
\end{proof}

\section{The cohomology of embedding tensors}\label{cohomology}
In this section, first we recall some basic
results involving representations and cohomologies of $3$-Leibniz algebra $(\mathcal{L},[\cdot,\cdot,\cdot]_{\mathcal{L}})$.
We construct a representation of the $3$-Leibniz algebra $(V;[\cdot,\cdot,\cdot]_{T})$ on the vector space $\g$ and define the cohomologies of an embedding tensor
on $3$-Lie algebras.

\begin{defi}
A {\bf representation} of a $3$-Leibniz algebra $(\mathcal{L},[\cdot,\cdot,\cdot]_{\huaL})$ is a quadruple $(V; l,m,r)$, where $V$ is a vector space, $l, m, r:\otimes^{2}\mathcal{L}\rightarrow \gl(V)$ are linear maps, such that for all $x_{i} \in \mathcal{L}, i=1,\cdots, 4$, the following equalities hold:
\begin{eqnarray}
\label{eq:=3LeibRep1}l(x_{1},x_{2})l(x_{3},x_{4})&=&l([x_{1},x_{2},x_{3}]_{\huaL},x_{4})+l(x_{3},[x_{1},x_{2},x_{4}]_{\huaL})+l(x_{3},x_{4})l(x_{1},x_{2}),\\
\label{eq:=3LeibRep2}l(x_{1},x_{2})m(x_{3},x_{4})&=&m([x_{1},x_{2},x_{3}]_{\huaL},x_{4})+m(x_{3},[x_{1},x_{2},x_{4}]_{\huaL})+m(x_{3},x_{4})l(x_{1},x_{2}),\\
\label{eq:=3LeibRep3}l(x_{1},x_{2})r(x_{3},x_{4})&=&r([x_{1},x_{2},x_{3}]_{\huaL},x_{4})+r(x_{3},[x_{1},x_{2},x_{4}]_{\huaL})+r(x_{3},x_{4})l(x_{1},x_{2}),\\
\label{eq:=3LeibRep4}m(x_{1},[x_{2},x_{3},x_{4}]_{\huaL})&=&r(x_{3},x_{4})m(x_{1},x_{2})+m(x_{2},x_{4})m(x_{1},x_{3})+l(x_{2},x_{3})m(x_{1},x_{4}),\\
\label{eq:=3LeibRep5}r(x_{1},[x_{2},x_{3},x_{4}]_{\huaL})&=&r(x_{3},x_{4})r(x_{1},x_{2})+m(x_{2},x_{4})r(x_{1},x_{3})+l(x_{2},x_{3})r(x_{1},x_{4}).
\end{eqnarray}
\end{defi}
An $n$-cochain on a 3-Leibniz algebra $(\mathcal{L},[\cdot,\cdot,\cdot]_{\mathcal{L}})$ with coefficients in a representation $(V; l,m,r)$ is a linear map
\begin{equation*}
f:\underbrace{((\otimes^{2}\huaL)\otimes\cdots\otimes(\otimes^{2}\mathcal{L}))}_{n-1}\otimes\mathcal{L}\longrightarrow V,\quad (n\geq1).
\end{equation*}
Denote the space of $n$-cochains by $\frkC_{\Le}^{n}(\mathcal{L};V)$. The coboundary operator $\partial:\frkC_{\Le}^{n}(\mathcal{L};V)\rightarrow \frkC_{\Le}^{n+1}(\mathcal{L};V)$ is given by
\begin{eqnarray*}
&&(\partial f)(X_{1},X_{2},\cdots,X_{n},z)\\
&=&\sum_{1\leq j<k\leq n}(-1)^{j}f(X_{1},\cdots,\hat{X_{j}},\cdots,X_{k-1},[x_{j},y_{j},x_{k}]_{\huaL}\otimes y_{k},X_{k+1},\cdots,X_{n},z)\\
&&+\sum_{1\leq j<k\leq n}(-1)^{j}f(X_{1},\cdots,\hat{X_{j}},\cdots,X_{k-1},x_{k}\otimes[x_{j},y_{j},y_{k}]_{\huaL},X_{k+1},\cdots,X_{n},z)\\
&&+\sum_{j=1}^{n}\Big((-1)^{j}f(X_{1},\cdots,\hat{X_{j}},\cdots,X_{n},[X_{j},z]_{\huaL})+(-1)^{j+1}l(X_{j})f(X_{1},\cdots,\hat{X_{j}},\cdots,X_{n},z)\Big)\\
&&+(-1)^{n+1}\Big(m(x_{n},z)f(X_{1},\cdots,X_{n-1},y_{n})+r(y_{n},z)f(X_{1},\cdots,X_{n-1},x_{n})\Big),
\end{eqnarray*}
where $X_{i}=x_{i}\otimes y_{i}\in \otimes^{2}\mathcal{L}, i=1, 2, \cdots, n$ and $z\in \mathcal{L}$. It was proved in \cite{Casas,Takhtajan1} that $\partial\circ \partial=0.$ Thus, $(\oplus_{n=1}^{+\infty}\frkC_{\Le}^{n}(\mathcal{L};V),\partial)$ is a cochain complex.

\begin{defi}
The {\bf cohomology} of the $3$-Leibniz algebra $(\mathcal{L},[\cdot,\cdot,\cdot]_{\mathcal{L}})$ with coefficients in $V$ is the cohomology of the cochain complex
$(\oplus_{n=1}^{+\infty}\frkC_{\Le}^{n}(\mathcal{L};V),\partial)$.  We denote the set of $n$-cocycles by $Z^n_{\Le}(\mathcal{L};V)$,  the set of $n$-coboundaries by $B^n_{\Le}(\mathcal{L};V)$ and the $n$-th cohomology group by  \begin{equation}\label{eq:cohomology}
H^n_{\Le}(\mathcal{L};V)=Z^n_{\Le}(\mathcal{L};V)/B^n_{\Le}(\mathcal{L};V).
\end{equation}
\end{defi}

\begin{lem}
  Let $T:V\rightarrow \g$ be an embedding tensor on the $3$-Lie algebra $(\g,[\cdot,\cdot,\cdot]_\g)$ with respect to the representation $(V;\rho)$.  Define linear maps $l_{T}, m_{T}, r_{T}:\otimes^{2}V\longrightarrow \gl(\g)$ by
  \begin{eqnarray}
\label{eq:=3LeibRep6}l_{T}(u,v)(x)&:=&[Tu, Tv, x]_{\g},\\
\label{eq:=3LeibRep7}m_{T}(u,v)(x)&:=&[Tu, x, Tv]_{\g}-T(\rho(Tu,x)v),\\
\label{eq:=3LeibRep8}r_{T}&:=&-m_{T},
\end{eqnarray}
for all $u,v \in V, x\in \g.$ Then $(\g;l_{T},m_{T},r_{T})$ is a representation of the $3$-Leibniz algebra $(V, [\cdot,\cdot,\cdot]_{T})$.
\end{lem}
\begin{proof}
For all $u_{1},u_{2},u_{3},u_{4} \in V$ and $x \in \g$, by \eqref{eq:=3-lie}, we have
\begin{eqnarray*}
&&\big(l_{T}([u_{1},u_{2},u_{3}]_T,u_{4})+l_{T}(u_{3},[u_{1},u_{2},u_{4}]_T)+l_{T}(u_{3},u_{4})l_{T}(u_{1},u_{2})\big)(x)\\
&=&[T[u_{1},u_{2},u_{3}]_{T},Tu_{4},x]_{\g}+[Tu_{3},T[u_{1},u_{2},u_{4}]_{T},x]_{\g}+[Tu_{3},Tu_{4},[Tu_{1},Tu_{2},x]_{\g}]_{\g}\\
&=&[[Tu_{1},Tu_{2},Tu_{3}]_{\g},Tu_{4},x]_{\g}+[Tu_{3},[Tu_{1},Tu_{2},Tu_{4}]_{\g},x]_{\g}+[Tu_{3},Tu_{4},[Tu_{1},Tu_{2},x]_{\g}]_{\g}\\
&=&[Tu_{1},Tu_{2},[Tu_{3},Tu_{4},x]_{\g}]_{\g}\\
&=&l_{T}(u_{1},u_{2})([Tu_{3},Tu_{4},x]_{\g})\\
&=&l_{T}(u_{1},u_{2})l_{T}(u_{3},u_{4})(x),
\end{eqnarray*}
which implies that \eqref{eq:=3LeibRep1} holds.

By \eqref{eq:=3-lie}, \eqref{rep1:=3-lie}, \eqref{ET1} and Proposition \ref{ET}, for all $u_{1},u_{2},u_{3},u_{4} \in V$ and $x \in \g$, we have
\begin{eqnarray*}
&&m_{T}(u_{3},u_{4})l_{T}(u_{1},u_{2})(x)+m_{T}([u_{1},u_{2},u_{3}]_T,u_{4})(x)+m_{T}(u_{3},[u_{1},u_{2},u_{4}]_T)(x)\\
&=&m_{T}(u_{3},u_{4}) ([Tu_{1},Tu_{2},x]_{\g})+[T[u_{1},u_{2},u_{3}]_T,x,Tu_{4}]_{\g}-T\big(\rho(T[u_{1},u_{2},u_{3}]_T,x)u_{4}\big)\\
&&+[Tu_{3},x,T[u_{1},u_{2},u_{4}]_T]_{\g}-T\big(\rho(Tu_{3},x)([u_{1},u_{2},u_{4}]_T)\big)\\
&=&[Tu_{3},[Tu_{1},Tu_{2},x]_{\g},Tu_{4}]_{\g}-T\big(\rho(Tu_{3},[Tu_{1},Tu_{2},x]_{\g})u_{4}\big)+[[Tu_{1},Tu_{2},Tu_{3}]_{\g},x,Tu_{4}]_{\g}\\
&&-T\big(\rho([Tu_{1},Tu_{2},Tu_{3}]_{\g},x)u_{4}\big)+[Tu_{3},x,[Tu_{1},Tu_{2},Tu_{4}]_{\g}]_{\g}-T\big(\rho(Tu_{3},x)\rho(Tu_{1},Tu_{2})u_{4}\big)\\
&=&[Tu_{1},Tu_{2},[Tu_{3},x,Tu_{4}]_{\g}]_{\g}-T\big(\rho(Tu_{1},Tu_{2})\rho(Tu_{3},x)u_{4}\big)\\
&=&[Tu_{1},Tu_{2},[Tu_{3},x,Tu_{4}]_{\g}]_{\g}-[Tu_{1},Tu_{2},T(\rho(Tu_{3},x)u_{4})]_{\g}\\
&=&l_{T}(u_{1},u_{2})\big([Tu_{3},x,Tu_{4}]_{\g}-T(\rho(Tu_{3},x)u_{4})\big)\\
&=&l_{T}(u_{1},u_{2})m_{T}(u_{3},u_{4})(x),
\end{eqnarray*}
which means that \eqref{eq:=3LeibRep2} holds. Similarly, we can prove that \eqref{eq:=3LeibRep3} is true.

For all $u_{1},u_{2},u_{3},u_{4} \in V$ and $x \in \g$, we get
\begin{eqnarray*}
&&r_{T}(u_{3},u_{4})m_{T}(u_{1},u_{2})(x)+m_{T}(u_{2},u_{4})m_{T}(u_{1},u_{3})(x)+l_{T}(u_{2},u_{3})m_{T}(u_{1},u_{4})(x)\\
&=&r_{T}(u_{3},u_{4})\big([Tu_{1},x,Tu_{2}]_{\g}-T(\rho(Tu_{1},x)u_{2})\big)+m_{T}(u_{2},u_{4})\big([Tu_{1},x,Tu_{3}]_{\g}-T(\rho(Tu_{1},x)u_{3})\big)\\
&&+l_{T}(u_{2},u_{3})\big([Tu_{1},x,Tu_{4}]_{\g}-T(\rho(Tu_{1},x)u_{4})\big)\\
&=&[[Tu_{1},x,Tu_{2}]_{\g},Tu_{3},Tu_{4}]_{\g}-T\big(\rho([Tu_{1},x,Tu_{2}]_{\g},Tu_{3})u_{4}\big)-[T(\rho(Tu_{1},x)u_{2}),Tu_{3},Tu_{4}]_{\g}\\
&&+T\big(\rho(T(\rho(Tu_{1},x)u_{2}),Tu_{3})u_{4}\big)+[Tu_{2},[Tu_{1},x,Tu_{3}]_{\g},Tu_{4}]_{\g}-T\big(\rho(Tu_{2},[Tu_{1},x,Tu_{3}]_{\g})u_{4}\big)\\
&&-[Tu_{2},T(\rho(Tu_{1},x)u_{3}),Tu_{4}]_{\g}+T\big(\rho(Tu_{2},T(\rho(Tu_{1},x)u_{3}))u_{4}\big)+[Tu_{2},Tu_{3},[Tu_{1},x,Tu_{4}]_{\g}]_{\g}\\
&&-[Tu_{2},Tu_{3},T(\rho(Tu_{1},x)u_{4})]_{\g}\\
&=&[[Tu_{1},x,Tu_{2}]_{\g},Tu_{3},Tu_{4}]_{\g}+[Tu_{2},[Tu_{1},x,Tu_{3}]_{\g},Tu_{4}]_{\g}+[Tu_{2},Tu_{3},[Tu_{1},x,Tu_{4}]_{\g}]_{\g}\\
&&-T\big(\rho(Tu_{2},[Tu_{1},x,Tu_{3}]_{\g})u_{4}\big)-T\big(\rho([Tu_{1},x,Tu_{2}]_{\g},Tu_{3})u_{4}\big)-[Tu_{2},Tu_{3},T(\rho(Tu_{1},x)u_{4})]_{\g}\\
&=&[Tu_{1},x,[Tu_{2},Tu_{3},Tu_{4}]_{\g}]_{\g}-T\big(\rho(Tu_{2},[Tu_{1},x,Tu_{3}]_{\g})u_{4}\big)-T\big(\rho([Tu_{1},x,Tu_{2}]_{\g},Tu_{3})u_{4}\big)\\
&&-T\big(\rho(Tu_{2},Tu_{3})\rho(Tu_{1},x)u_{4}\big)\\
&=&[Tu_{1},x,[Tu_{2},Tu_{3},Tu_{4}]_{\g}]_{\g}-T\big(\rho(Tu_{1},x)\rho(Tu_{2},Tu_{3})u_{4}\big)\\
&=&[Tu_{1},x,T[u_{2},u_{3},u_{4}]_T]_{\g}-T\big(\rho(Tu_{1},x)([u_{2},u_{3},u_{4}]_T)\big)\\
&=&m_{T}(u_{1},[u_{2},u_{3},u_{4}]_T)(x),
\end{eqnarray*}
which indicates that \eqref{eq:=3LeibRep4} holds. Similarly, we can show that  \eqref{eq:=3LeibRep5} holds. Therefore, $(\g;l_{T},m_{T},r_{T})$ is a representation of the $3$-Leibniz algebra $(V, [\cdot,\cdot,\cdot]_{T})$.
\end{proof}

Let $\partial_{T}:\frkC_{\Le}^{n}(V;\g)\rightarrow  \frkC_{\Le}^{n+1}(V;\g),~(n\geq 1)$ be the coboundary operator of the 3-Leibniz algebra $(V, [\cdot,\cdot,\cdot]_{T})$ with coefficients in the representation $(\g;l_{T},m_{T},r_{T})$. More precisely, for all $\theta \in \Hom (\underbrace{(\otimes^{2} V)\otimes \cdots\otimes (\otimes^{2}V}_{n-1})\otimes V,\g)$, $U_i=u_i\otimes v_i\in \otimes^2V,~ i=1,2,\cdots,n$ and $w\in V,$
we have
\begin{eqnarray}
\label{eq:=cohoET}&&\partial_{T}\theta(U_{1},U_{2},\cdots,U_{n},w)\\
\nonumber&=&\sum_{1\leq j<k\leq n}(-1)^{j}\theta(U_{1},\cdots,\hat{U_{j}},\cdots,U_{k-1},[u_{j},v_{j},u_{k}]_{T}\otimes v_{k},U_{k+1},\cdots,U_{n},w)\\
\nonumber&&+\sum_{1\leq j<k\leq n}(-1)^{j}\theta(U_{1},\cdots,\hat{U_{j}},\cdots,U_{k-1},u_{k}\otimes[u_{j},v_{j},v_{k}]_{T},U_{k+1},\cdots,U_{n},w)\\
\nonumber&&+\sum_{j=1}^{n}\Big((-1)^{j}\theta(U_{1},\cdots,\hat{U_{j}},\cdots,U_{n},[U_{j},w]_{T})+(-1)^{j+1}l_{T}(U_{j})\theta(U_{1},\cdots,\hat{U_{j}},\cdots,U_{n},w)\Big)\\
\nonumber&&+(-1)^{n+1}\Big(m_{T}(u_{n},w)\theta(U_{1},\cdots,U_{n-1},v_{n})+r_{T}(v_{n},w)\theta(U_{1},\cdots,U_{n-1},u_{n})\Big).
\end{eqnarray}

It is obvious that $\theta \in \Hom(V,\g)$ is closed if and only if
\begin{eqnarray*}
&&[Tu,Tv,\theta w]_{\g}+[\theta u,Tv,Tw]_{\g}+[Tu,\theta v,Tw]_{\g}\\
&=&T(\rho(\theta u,Tv)w)+T(\rho(Tu,\theta v)w)+\theta(\rho(Tu,Tv)w).
\end{eqnarray*}

For any $\mathfrak{X} \in \g\wedge\g, v \in V$, define $\delta(\mathfrak{X}):V\longrightarrow \g$ by
\begin{eqnarray}\label{0-cochain}
\delta(\mathfrak{X})v=T\rho(\mathfrak{X})v-[\mathfrak{X},Tv]_{\g}.
\end{eqnarray}
\begin{pro}\label{cocycle-1-em}
Let $T$ be an embedding tensor on the $3$-Lie algebra $(\g,[\cdot,\cdot,\cdot]_\g)$ with respect to the representation $(V;\rho)$. Then $\delta(\mathfrak{X})$ is a $1$-cocycle on the $3$-Leibniz algebra $(V,[\cdot,\cdot,\cdot]_{T})$ with coefficients in $(\g;l_{T},m_{T},r_{T})$.
\end{pro}
\begin{proof}
For any $u,v,w \in V$, by \eqref{eq:=3-lie}, \eqref{rep1:=3-lie} and the fact that T is an embedding tensor, we have
\begin{eqnarray*}
(\partial_{T}\delta(\mathfrak{X}))(u,v,w)&=&-\delta(\mathfrak{X})([u,v,w]_{T})+[Tu,Tv,\delta(\mathfrak{X})w]_{\g}+[\delta(\mathfrak{X})u,Tv,Tw]_{\g}\\
&&-T(\rho(\delta(\mathfrak{X})u,Tv)w)+[Tu,\delta(\mathfrak{X})v,Tw]_{\g}-T(\rho(Tu,\delta(\mathfrak{X})v)w)\\
&=&-T\rho(\mathfrak{X})([u,v,w]_{T})+[\mathfrak{X},T[u,v,w]_{T}]_{\g}+[Tu,Tv,T\rho(\mathfrak{X})w]_{\g}\\
&&-[Tu,Tv,[\mathfrak{X},Tw]_{\g}]_{\g}+[T\rho(\mathfrak{X})u,Tv,Tw]_{\g}-[[\mathfrak{X},Tu]_{\g},Tv,Tw]_{\g}\\
&&-T(\rho(T\rho(\mathfrak{X})u,Tv)w)+T(\rho([\mathfrak{X},Tu]_{\g},Tv)w)+[Tu,T\rho(\mathfrak{X})v,Tw]_{\g}\\
&&-[Tu,[\mathfrak{X},Tv]_{\g},Tw]_{\g}-T(\rho(Tu,T\rho(\mathfrak{X})v)w)+T(\rho(Tu,[\mathfrak{X},Tv]_{\g})w)\\
&=&-T\big(\rho(\mathfrak {X})\rho(Tu,Tv)w\big)+T\big(\rho(Tu,Tv)\rho(\mathfrak {X})w\big)+T\big(\rho(T\rho(\mathfrak {X})u,Tv)w\big)\\
&&-T\big(\rho(T\rho(\mathfrak {X})u,Tv)w\big)+T\big(\rho([\mathfrak{X},Tu]_{\g},Tv)w\big)+T\big(\rho(Tu,T\rho(\mathfrak{X})v)w\big)\\
&&-T\big(\rho(Tu,T\rho(\mathfrak{X})v)w\big)+T\big(\rho(Tu,[\mathfrak{X},Tv]_{\g})w\big)\\
&=&0.
\end{eqnarray*}
Thus, we deduce that $\partial_{T}\delta(\mathfrak{X})=0.$
\end{proof}
Now we define the cohomology theory of an embedding tensor on the $3$-Lie algebra $(\g,[\cdot,\cdot,\cdot]_\g)$ with respect to the representation $(V;\rho)$.

Let $T$ be an embedding tensor on the $3$-Lie algebra $(\g,[\cdot,\cdot,\cdot]_\g)$ with respect to the representation $(V;\rho)$. Define the set of $k$-cochains by
\begin{equation*}
\frak C_{T}^{k}(V;\g)=
\begin{cases}
\frak C_{\Le}^{k-1}(V;\g),&k\geq2,\\
\g\wedge \g,& k=1.
\end{cases}
\end{equation*}
Define $\dM^{\mathsf{T}}:\frak C_{T}^{k}(V;\g)\rightarrow \frak C_{T}^{k+1}(V;\g)$ by
\begin{equation*}
\dM^{\mathsf{T}}=
\begin{cases}
\partial_{T},& k\geq2,\\
\delta,& k=1.
\end{cases}
\end{equation*}

\begin{thm}
 $(\mathop{\oplus}\limits_{k=1}^{\infty} \mathfrak C_{T}^{k}(V;\g),\dM^{\mathsf{T}})$ is a cochain complex.
\end{thm}
\begin{proof}
It follows from Proposition \ref{cocycle-1-em} and the fact that $\partial_{T}$ is the corresponding coboundary operator of the $3$-Leibniz algebra $(V, [\cdot,\cdot,\cdot]_{T} )$ with coefficients in the representation $(\g;l_{T},m_{T},r_{T})$
directly.
\end{proof}

\begin{defi}
The cohomology of the cochain complex $(\mathop{\oplus}\limits_{k=1}^{\infty} \mathfrak C_{T}^{k}(V;\g),\dM^{\mathsf{T}})$ is taken to be the {\bf cohomology for the embedding tensor $T$}.
Denote the set of $k$-cocycles by $\huaZ^{k}(T)$ and the set of $k$-coboundaries by $\huaB^{k}(T)$. The $k$-th cohomology group of the embedding tensor $T$ is denoted by
\begin{eqnarray*}
\huaH^{k}(T)=\huaZ^{k}(T)/\huaB^{k}(T),\quad k\geq 1.
\end{eqnarray*}
\end{defi}

At the end of this section, we give the relationship between the differential $l_1^{T}$ defined by ~\eqref{eq:=L1} using the Maurer-Cartan element $T$ of the Lie $3$-algebra $(C^*(V,\g),\{\cdot,\cdot,\cdot\})$ and the coboundary operator $\dM^{\mathsf{T}}$ of the embedding tensor $T$.
\begin{thm}\label{thm:=L5}
Let $T$ be an embedding tensor on the $3$-Lie algebra $(\g,[\cdot,\cdot,\cdot]_{\g})$ with respect to the representation $(V;\rho)$. Then we have
\begin{eqnarray}\label{eq:=L5}
\dM^{\mathsf{T}} \theta=(-1)^{n-1}l_{1}^{T}\theta,\quad \forall \theta \in \Hom(\underbrace{(\otimes^{2}V)\otimes\cdots \otimes(\otimes^{2}V}_{n-1})\otimes V,\g),n=1,2,\cdots.
\end{eqnarray}
\end{thm}
\begin{proof}
For all $U_{i}=u_{i}\otimes v_{i}\in \otimes^{2}V, i=1,2,\cdots,n$ and $w \in V$, by Lemma \ref{lem-equation-1} and \eqref{eq:=L1}, we get
\begin{eqnarray*}
&&2l_{1}^{T}\theta(U_{1},U_{2},\cdots,U_{n},w)\\
&=&\{T,T,\theta\}(U_{1},U_{2},\cdots,U_{n},w)\\
&=&{[[[\mu\boxplus\rho,T]_{\Le},T]_{\Le},\theta]}_{\Le}(U_{1},U_{2},\cdots,U_{n},w)\\
&=&{[[\mu\boxplus\rho,T]_{\Le},T]}_{\Le}\big(\theta(U_{1},\cdots,U_{n-1},u_{n})\otimes v_{n},w\big)\\
&&+{[[\mu\boxplus\rho,T]_{\Le},T]}_{\Le}\big(u_{n}\otimes\theta(U_{1},\cdots,U_{n-1},v_{n}),w\big)\\
&&+\sum_{i=1}^{n}(-1)^{n-1}(-1)^{i-1}{[[\mu\boxplus\rho,T]_{\Le},T]}_{\Le}(U_{i},\theta(U_{1},\cdots,\hat{U_{i}},\cdots,U_{n},w))\\
&&-(-1)^{n-1}\sum_{k=1}^{n-1}\sum_{i=1}^{k}(-1)^{i+1}\theta\big(U_{1},\cdots,\hat{U_{i}},\cdots,U_{k},{[[\mu\boxplus\rho,T]_{\Le},T]}_{\Le}(U_{i},u_{k+1})\otimes v_{k+1},\\
&&U_{k+2},\cdots,U_{n},w\big)\\
&&-(-1)^{n-1}\sum_{k=1}^{n-1}\sum_{i=1}^{k}(-1)^{i+1}\theta\big(U_{1},\cdots,\hat{U_{i}},\cdots,U_{k},u_{k+1}\otimes{[[\mu\boxplus\rho,T]_{\Le},T]}_{\Le}(U_{i},v_{k+1}),\\
&&U_{k+2},\cdots,U_{n},w\big)\\
&&-(-1)^{n-1}\sum_{i=1}^{n}(-1)^{i+1}\theta\big(U_{1},\cdots,\hat{U}_{i},\cdots,U_{n},{[[\mu\boxplus\rho,T]_{\Le},T]}_{\Le}(U_{i},w)\big)\\
&&=2(-1)^{n-1}\dM^{\mathsf{T}} \theta,
\end{eqnarray*}
which implies that $\dM^{\mathsf{T}} \theta=(-1)^{n-1}l_{1}^{T}\theta$. The proof is finished.
\end{proof}

\section{Deformations of embedding tensors}\label{defor}

In this section, we study formal deformations of embedding tensors on the $3$-Lie algebra $(\g,[\cdot,\cdot,\cdot]_{\g})$ with respect to the representation $(V;\rho)$ and show that if two formal deformations of an embedding tensor on a $3$-Lie algebra are equivalent, then
their infinitesimals are in the same cohomology class. Moreover, we also study the extendability of an order $n$ deformation to an order $n+1$ deformation of embedding tensors on $3$-Lie algebras.

\subsection{Formal deformations of embedding tensors on a 3-Lie algebra}
Let $\mathbb{K}[[t]]$ be the ring of power series in one variable $t$. Let $(\g,[\cdot,\cdot,\cdot]_{\g})$ be a $3$-Lie algebra over $\mathbb{K}$ and $\mathbb{V}[[t]]$ denote the vector space of formal power series in $t$ with coefficients in $V$. Then there is a $3$-Lie algebra structure over the ring $\mathbb{K}[[t]]$ on $\g[[t]]$ given by
\begin{eqnarray*}
\bigg[\sum_{i\geq 0}x_{i}t^{i},\sum_{j\geq 0}y_{j}t^{j},\sum_{k\geq 0}z_{k}t^{k}\bigg]:=\sum_{s\geq 0}\sum_{i+j+k=s}[x_{i},y_{j},z_{k}]_{\g}t^{s},\quad \forall x_{i},y_{j},z_{k} \in \g.
\end{eqnarray*}

For any representation $(V;\rho)$ of the $3$-Lie algebra $(\g,[\cdot,\cdot,\cdot]_{\g})$, there is a natural action of $\mathbb{\g}[[t]]$ on the $\mathbb{K}[[t]]$-module $\mathbb{V}[[t]]$, which is given by
\begin{eqnarray*}
\rho(\sum_{i\geq 0}x_{i}t^{i},\sum_{j\geq 0}y_{j}t^{j})(\sum_{k\geq 0}v_{k}t^{k})=\sum_{s\geq 0}\sum_{i+j+k=s}\rho(x_{i},y_{j})v_{k}t^{s}, \quad \forall x_{i},y_{j} \in \g, v_{k} \in V.
\end{eqnarray*}

 Consider a power series
\begin{equation*}
T_{t}=\sum_{i\geq 0}\tau_{i}t^{i},\quad \tau_{i}\in \Hom_{\mathbb{K}}(V,\g),
\end{equation*}
that is, $T_{t}\in \Hom_{\mathbb{K}}(V,\g)[[t]]$. Extended it to be a $\mathbb{K}[[t]]$-module map from $\mathbb{V}[[t]]$ to $\mathbb{\g}[[t]]$ which is still denoted by $T_{t}$.
\begin{defi}
Let $T$ be an embedding tensor on the $3$-Lie algebra $(\g,[\cdot,\cdot,\cdot]_{\g})$ with respect to the representation $(V;\rho)$.
If a power series
\begin{equation*}
T_{t}=\sum_{i\geq 0}\tau_{i}t^{i},\quad \tau_{i}\in \Hom_{\mathbb{K}}(V,\g),
\end{equation*}
where $\tau_{0}=T$ satisfies
\begin{equation}\label{eq:=FD1}
[T_{t}u,T_{t}v,T_{t}w]_{\g}=T_{t}\big(\rho(T_{t}u,T_{t}v)w\big),\quad \forall u,v,w\in V,
\end{equation}
then $T_{t}$ is called a {\bf formal deformation} of an embedding tensor $T$.
\end{defi}

Let $T$ be an embedding tensor on the $3$-Lie algebra $(\g,[\cdot,\cdot,\cdot]_{\g})$ with respect to the representation $(V;\rho)$ and $T_{t}$ be a formal deformation. For any $u,v,w\in V$, we have the following equation:
\begin{equation}\label{FD2}
\sum_{s\geq 0}\sum_{i+j+k=s}[\tau_{i}u,\tau_{j}v,\tau_{k}w]_{\g}t^{s}=\sum_{s\geq 0}\sum_{i+j+k=s}\tau_{k}(\rho(\tau_{i}u,\tau_{j}v)w)t^{s}.
\end{equation}

For $s=1$, \eqref{FD2} is equivalent to
\begin{eqnarray*}
&&[Tu,Tv,\tau_{1}w]_{\g}+[\tau_{1} u,Tv,Tw]_{\g}+[Tu,\tau_{1}v,Tw]_{\g}\\
&=&\tau_{1}(\rho(Tu,Tv)w)+T(\rho(\tau_{1} u,Tv)w)+T(\rho(Tu,\tau_{1} v)w).
\end{eqnarray*}
Then we have the following results.
\begin{pro}\label{pro:=FD3} Let $T_{t}=\sum_{i\geq 0}\tau_{i}t^{i}$ be a formal deformation of an embedding tensor on the $3$-Lie algebra $(\g,[\cdot,\cdot,\cdot]_{\g})$ with respect to the representation $(V;\rho)$. Then $\tau_{1}$ is a $2$-cocycle for the embedding tensor $T$, that is $\dM^{\mathsf{T}}\tau_{1}=0$.
\end{pro}
\begin{defi}
Let $T$ be an embedding tensor on the $3$-Lie algebra $(\g,[\cdot,\cdot,\cdot]_{\g})$ with respect to the representation $(V;\rho)$. Then $\tau_{1}$ given by Proposition $\ref{pro:=FD3}$ is called the {\bf infinitesimal} of the formal deformation $T_{t}=\sum_{i\geq 0}\tau_{i}t^{i}$ of $T$.
\end{defi}
In the sequel, we discuss equivalent formal deformations.
\begin{defi}
Two formal deformations of an embedding tensor $T=\tau_{0}=\tilde{\tau}_{0}$ on a $3$-Lie algebra $T_{t}=\sum_{i\geq 0}\tau_{i}t^{i}$ and $\tilde{T}_{t}=\sum_{i\geq 0}\tilde{\tau}_{i}t^{i}$ are {\bf equivalent} if there exist $\mathfrak{X} \in \wedge^{2}\g$, $\phi_{i}\in \gl(\g),\psi_{i}\in \gl(V), i\geq 2$, such that for
\begin{equation}\label{eq:=FD4-0}
\phi_{t}=\Id_{\g}+t\ad_{\mathfrak{X}}+\sum_{i\geq 2}t^{i}\phi_{i},\quad \psi_{t}=\Id_{V}+t\rho(\mathfrak{X})+\sum_{i\geq 2}t^{i}\psi_{i},
\end{equation}
the following conditions hold:
\begin{eqnarray}
\label{eq:=FD4-1}\phi_{t}([x,y,z]_{\g})&=&[\phi_{t}(x),\phi_{t}(y),\phi_{t}(z)]_{\g}, \quad \forall x,y,z \in \g,\\
\label{eq:=FD4-2}\rho(\phi_{t}(x),\phi_{t}(y))\psi_{t}(u)&=&\psi_{t}(\rho(x,y)u), \quad \forall u\in V,\\
\label{eq:=FD4-3}T_{t}\circ\psi_{t}&=&\phi_{t}\circ\tilde{T}_{t}.
\end{eqnarray}
\end{defi}
\begin{thm}
Let $T_{t}$ and $\tilde{T}_{t}$ be two equivalent formal deformations on the $3$-Lie algebra $(\g,[\cdot,\cdot,\cdot]_{\g})$ with respect to the representation $(V;\rho)$. Then their infinitesimals are in the same cohomology class in $\huaH^{2}(T)$.
\end{thm}
\begin{proof}
By $T_{t}\circ\psi_{t}=\phi_{t}\circ\tilde{T}_{t}$, we deduce that
\begin{equation*}
\tilde{\tau}_{1}u=\tau_{1}u+T\rho(\mathfrak{X})u-[\mathfrak{X},Tu]_{\g}=\tau_{1}u+(\dM^{\mathsf{T}}\mathfrak{X})u,
\end{equation*}
which means that $\tilde{\tau}_{1}$ and $\tau_{1}$ are in the same cohomology class.
\end{proof}

\subsection{Order $n$ deformations of embedding tensors on a 3-Lie algebra}
\begin{defi}
Let $T:V\rightarrow \g$ be an embedding tensor on the $3$-Lie algebra $(\g,[\cdot,\cdot,\cdot]_{\g})$ with respect to a representation $(V;\rho)$. If $T_{t}=\sum_{i=0}^{n}\tau_{i}t^{i}$ with $\tau_{0}=T, \tau_{i}\in \Hom_{\mathbb{K}} (V,\g), i=1,\cdots,n,$ defines a $\mathbb{K}[[t]]/(t^{n+1})$-module map from $\V[[t]]/(t^{n+1})$ to the $3$-Lie algebra $\g[[t]]/(t^{n+1})$ and satisfies
\begin{equation}\label{eq:=order n1}
[T_{t}u,T_{t}v,T_{t}w]_{\g}=T_{t}\Big(\rho(T_{t}u,T_{t}v)w\Big),\quad \forall u,v,w\in V.
\end{equation}
Then $T_{t}$ is called an {\bf order $n$ deformation} of the embedding tensor $T.$

\end{defi}
For all $0\leq s\leq n,~u,v,w\in V$, we have
\begin{eqnarray}\label{eq:=order n2}
\sum_{i+j+k=s\atop
 0\leq i,j,k\leq s}\Big([\tau_{i}u,\tau_{j}v,\tau_{k}w]_{\g}-\tau_{k}(\rho(\tau_{i}u,\tau_{j}v)w)\Big)=0.
\end{eqnarray}

\begin{defi}
Let $T_{t}=\sum_{i=0}^{n}\tau_{i}t^{i}$ be an order $n$ deformation of an embedding tensor on a $3$-Lie algebra $(\g,[\cdot,\cdot,\cdot]_{\g})$ with respect to a representation $(V;\rho)$. If there exists a $1$-cochain $\tau_{n+1}\in \Hom_{\mathbb{K}}(V,\g)$ such that $\hat{T}_{t}=T_{t}+\tau_{n+1}t^{n+1}$ is an order $n+1$ deformation of the embedding tensor $T$, then $T_{t}$ is called {\bf extendable}.
\end{defi}
Let $T_{t}=\sum_{i=0}^{n}\tau_{i}t^{i}$ be an order $n$ deformation of an embedding tensor $T$ on the $3$-Lie algebra $(\g,[\cdot,\cdot,\cdot]_{\g})$ with respect to the representation $(V;\rho)$. Define $\Ob \in \frak C_{T}^{3}(V;\g)$ by
\begin{equation}\label{eq:=Obt1}
\Ob(u,v,w)=\sum_{i+j+k=n+1\atop
0\leq i,j,k\leq n}\Big([\tau_{i}u,\tau_{j}v,\tau_{k}w]_{\g}-\tau_{k}(\rho(\tau_{i}u,\tau_{j}v)w)\Big).
\end{equation}
Then we deduce the proposition as follows.
\begin{pro}
With the above notations, $\Ob$ is a $3$-cocycle, that is $\dM^{\mathsf{T}}\Ob=0$.
\end{pro}
\begin{proof}
According to the bracket given by \eqref{lie2-bracket}, we have
\begin{eqnarray*}
&&\{\tau_i,\tau_j,\tau_k\}(u,v,w)\\
&=&[[[\mu\boxplus\rho,\tau_i]_{\Le},\tau_j]_{\Le},\tau_k]_{\Le}(u,v,w)\\
&=&[[\mu\boxplus\rho,\tau_i]_{\Le},\tau_j]_{\Le}(\tau_{k}u,v,w)+[[\mu\boxplus\rho,\tau_i]_{\Le},\tau_j]_{\Le}(u,\tau_{k}v,w)\\
&&+[[\mu\boxplus\rho,\tau_i]_{\Le},\tau_j]_{\Le}(u,v,\tau_{k}w)-\tau_{k}[[\mu\boxplus\rho,\tau_i]_{\Le},\tau_j]_{\Le}(u,v,w)\\
&=&[\mu\boxplus\rho,\tau_i]_{\Le}(\tau_{k}u,\tau_{j}v,w)+[\mu\boxplus\rho,\tau_i]_{\Le}(\tau_{k}u,v,\tau_{j}w)-\tau_{j}[\mu\boxplus\rho,\tau_i]_{\Le}(\tau_{k}u,v,w)\\
&&+[\mu\boxplus\rho,\tau_i]_{\Le}(\tau_{j}u,\tau_{k}v,w)+[\mu\boxplus\rho,\tau_i]_{\Le}(u,\tau_{k}v,\tau_{j}w)-\tau_{j}[\mu\boxplus\rho,\tau_i]_{\Le}(u,\tau_{k}v,w)\\
&&+[\mu\boxplus\rho,\tau_i]_{\Le}(\tau_{j}u,v,\tau_{k}w)+[\mu\boxplus\rho,\tau_i]_{\Le}(u,\tau_{j}v,\tau_{k}w)-\tau_{j}[\mu\boxplus\rho,\tau_i]_{\Le}(u,v,\tau_{k}w)\\
&&-\tau_{k}[\mu\boxplus\rho,\tau_i]_{\Le}(\tau_{j}u,v,w)-\tau_{k}[\mu\boxplus\rho,\tau_i]_{\Le}(u,\tau_{j}v,w)-\tau_{k}[\mu\boxplus\rho,\tau_i]_{\Le}(u,v,\tau_{j}w)\\
&&+\tau_{k}\tau_{j}[\mu\boxplus\rho,\tau_i]_{\Le}(u,v,w)\\
&=&[\tau_{k}u,\tau_{j}v,\tau_{i}w]_{\g}-\tau_{i}\rho(\tau_{k}u,\tau_{j}v)w+[\tau_{k}u,\tau_{i}v,\tau_{j}w]_{\g}-\tau_{j}\rho(\tau_{k}u,\tau_{i}v)w\\
&&+[\tau_{j}u,\tau_{k}v,\tau_{i}w]_{\g}-\tau_{i}\rho(\tau_{j}u,\tau_{k}v)w-\tau_{j}\rho(\tau_{i}u,\tau_{k}v)w+[\tau_{j}u,\tau_{i}v,\tau_{k}w]_{\g}\\
&&+[\tau_{i}u,\tau_{j}v,\tau_{k}w]_{\g}-\tau_{k}\rho(\tau_{j}u,\tau_{i}v)w-\tau_{k}\rho(\tau_{i}u,\tau_{j}v)w+[\tau_{i}u,\tau_{k}v,\tau_{j}w]_{\g}.
\end{eqnarray*}
This gives
\begin{eqnarray}
\Ob=\frac{1}{6}\sum_{i+j+k=n+1\atop 0\leq i,j,k \leq n}\{\tau_{i},\tau_{j},\tau_{k}\}=\frac{1}{6}\sum_{i+j+k=n+1\atop 1\leq i,j,k \leq n}\{\tau_{i},\tau_{j},\tau_{k}\}+\frac{1}{2}\sum_{i+j=n+1\atop 1\leq i,j \leq n}\{T,\tau_{i},\tau_{j}\}.
\end{eqnarray}
Since $T_t$ is an order
$n$ deformation  of the  embedding tensor $T$, \eqref{eq:=order n2} is equivalent to
\begin{eqnarray}\label{eq:=order n3}
\label{deformation-2}-\frac{1}{2}\{T,T,\tau_{s}\}&=&\frac{1}{6}\sum\limits_{i+j+k=s\atop 0\leq i,j,k\leq s-1}\{\tau_i,\tau_j,\tau_k\}\\
\nonumber&=&\frac{1}{6}\sum\limits_{i+j+k=s\atop 1\leq i,j,k\leq s-1}\{\tau_i,\tau_j,\tau_k\}+\frac{1}{2}\sum\limits_{i+j=s\atop 1\leq i,j\leq s-1}\{T,\tau_i,\tau_j\},\quad 0\leq s\leq n.
\end{eqnarray}
By Theorem \ref{thm:=L5} and \eqref{eq:=L1}, we have
\begin{eqnarray*}
\dM^{\mathsf{T}}\Ob&=&-\frac{1}{2}\{T,T,\Ob\}\\
             &=&-\frac{1}{12}\sum\limits_{i+j+k=n+1\atop 1\leq i,j,k\leq n}\{T,T,\{\tau_i,\tau_j,\tau_k\}\}-\frac{1}{4}\sum\limits_{i+j=n+1\atop 1\leq i,j\leq n}\{T,T,\{T,\tau_i,\tau_j\}\}\\
             &=&-\frac{1}{12}\sum\limits_{i+j+k=n+1\atop 1\leq i,j,k\leq n}\{\{\tau_i,\tau_j,\tau_k\},T,T\}-\frac{1}{4}\sum\limits_{i+j=n+1\atop 1\leq i,j\leq n}\{\{T,\tau_i,\tau_j\},T,T\}\\
               &=&\frac{1}{6}\sum\limits_{i+j+k=n+1\atop 1\leq i,j,k\leq n}\Big(\{\{\tau_i,\tau_j,T\},\tau_k,T\}+\{\{\tau_i,\tau_k,T\},\tau_j,T\}+\{\{\tau_j,\tau_k,T\},\tau_i,T\}\Big)\\
               &&+\frac{1}{12}\sum\limits_{i+j+k=n+1\atop 1\leq i,j,k\leq n}\Big(\{\{\tau_i,T,T\},\tau_j,\tau_k\}+\{\{\tau_j,T,T\},\tau_i,\tau_k\}+\{\{\tau_k,T,T\},\tau_i,\tau_j\}\Big)\\
               &&+\frac{1}{4}\sum\limits_{i+j=n+1\atop 1\leq i,j\leq n}\Big(\{\{T,T,\tau_i\},\tau_j,T\}+\{\{T,T,\tau_j\},\tau_i,T\}\Big)\\
               &=&\frac{1}{2}\sum\limits_{i+j+k=n+1\atop 1\leq i,j,k\leq n}\{\{\tau_i,\tau_j,T\},\tau_k,T\}+\frac{1}{4}\sum\limits_{i+j+k=n+1\atop 1\leq i,j,k\leq n}\{\{\tau_i,T,T\},\tau_j,\tau_k\}+\frac{1}{2}\sum\limits_{i+j=n+1\atop 1\leq i,j\leq n}\{\{T,T,\tau_i\},\tau_j,T\}\\
               &=&\frac{1}{2}\sum\limits_{i+j+k=n+1\atop 1\leq i,j,k\leq n}\{\{\tau_i,\tau_j,T\},\tau_k,T\}-\frac{1}{12}\sum\limits_{i'+i''+i'''+j+k=n+1\atop 1\leq i',i'',i''',j,k\leq n}\{\{\tau_{i'},\tau_{i''},\tau_{i'''}\},\tau_j,\tau_k\}\\
               &&-\frac{1}{4}\sum\limits_{i'+i''+j+k=n+1\atop 1\leq i',i'',j,k\leq n}\{\{T,\tau_{i'},\tau_{i''}\},\tau_j,\tau_k\}-\frac{1}{6}\sum\limits_{i'+i''+i'''+j=n+1\atop 1\leq i',i'',i''',j\leq n}\{\{\tau_{i'},\tau_{i''},\tau_{i'''}\},\tau_j,T\}\\
               &&-\frac{1}{2}\sum\limits_{i'+i''+j=n+1\atop 1\leq i',i'',j\leq n}\{\{T,\tau_{i'},\tau_{i''}\},\tau_j,T\}\\
               &=&-\frac{1}{12}\sum\limits_{i'+i''+i'''+j+k=n+1\atop 1\leq i',i'',i''',j,k\leq n}\{\{\tau_{i'},\tau_{i''},\tau_{i'''}\},\tau_j,\tau_k\}-\frac{1}{4}\sum\limits_{i'+i''+j+k=n+1\atop 1\leq i',i'',j,k\leq n}\{\{T,\tau_{i'},\tau_{i''}\},\tau_j,\tau_k\}\\
               &&-\frac{1}{6}\sum\limits_{i'+i''+i'''+j=n+1\atop 1\leq i',i'',i''',j\leq n}\{\{\tau_{i'},\tau_{i''},\tau_{i'''}\},\tau_j,T\}\\
               &=&-\frac{1}{4}\sum\limits_{i'+i''+j+k=n+1\atop 1\leq i',i'',j,k\leq n}\{\{T,\tau_{i'},\tau_{i''}\},\tau_j,\tau_k\}-\frac{1}{6}\sum\limits_{i'+i''+i'''+j=n+1\atop 1\leq i',i'',i''',j\leq n}\{\{\tau_{i'},\tau_{i''},\tau_{i'''}\},\tau_j,T\}\\
              &=&-\frac{1}{12}\sum\limits_{i'+i''+j+k=n+1\atop 1\leq i',i'',j,k\leq n}\Big(\{\{T,\tau_{i'},\tau_{i''}\},\tau_j,\tau_k\}+\{\{\tau_{i'},T,\tau_{i''}\},\tau_j,\tau_k\}+\{\{\tau_{i'},\tau_{i''},T\},\tau_j,\tau_k\}\Big)\\
              &&-\frac{1}{12}\sum\limits_{i'+i''+i'''+j=n+1\atop 1\leq i',i'',i''',j\leq n}\Big(\{\{\tau_{i'},\tau_{i''},\tau_{i'''}\},\tau_j,T\}+\{\{\tau_{i'},\tau_{i''},\tau_{i'''}\},T,\tau_j\}\Big)\\
               &=&0.
\end{eqnarray*}
In this way, we obtain that the $3$-cochain $\Ob$ is a $3$-cocycle.
\end{proof}
\begin{defi}
Let $T_{t}=\sum_{i=0}^{n}\tau_{i}t^{i}$ be an order $n$ deformation of an embedding tensor on the $3$-Lie algebra $(\g,[\cdot,\cdot,\cdot]_{\g})$ with respect to the representation $(V;\rho)$. The cohomology class $[\Ob] \in \huaH^{3}(T)$ is called the {\bf obstruction class} of $T_{t}$ being extendable.
\end{defi}
\begin{thm}
Let $T_{t}=\sum_{i=0}^{n}\tau_{i}t^{i}$ be an order $n$ deformation of an embedding tensor on the $3$-Lie algebra $(\g, [\cdot,\cdot,\cdot]_{\g})$ with respect to the representation $(V ; \rho)$. Then $T_{t}$ is extendable if and only if the obstruction $[\Ob]$ is trivial.
\end{thm}
\begin{proof}
Suppose that an order $n$ deformation $T_{t}$ of the embedding tensor $T$ extends to an order $n+1$ deformation, then we have
\begin{equation*}
\sum_{i+j+k=n+1}([\tau_{i}u,\tau_{j}v,\tau_{k}w]_{\g}-\tau_{k}(\rho(\tau_{i}u,\tau_{j}v)w))=0,
\end{equation*}
which implies
\begin{eqnarray*}
&&\sum_{i+j+k=n+1}([\tau_{i}u,\tau_{j}v,\tau_{k}w]_{\g}-\tau_{k}(\rho(\tau_{i}u,\tau_{j}v)w))\\
&=&-\tau_{n+1}(\rho(Tu,Tv)w)+[Tu,Tv,\tau_{n+1}w]_{\g}+[\tau_{n+1}u,Tv,Tw]_{\g}-T(\rho(\tau_{n+1}u,Tv)w)\\
&&+[Tu,\tau_{n+1}v,Tw]_{\g}-T(\rho(Tu,\tau_{n+1}v)w).
\end{eqnarray*}
That is $\Ob=\dM^{\mathsf{T}}\tau_{n+1}$, thus the obstruction  class $[\Ob]$ is trivial.

Conversely, if the obstruction class $[\Ob]$ is trivial, suppose that $\Ob=\dM^{\mathsf{T}}\tau_{n+1}$ for some linear map $\tau_{n+1}\in \Hom(V,\g)$, set $\hat{T}_{t}=T_{t}+\tau_{n+1}t^{n+1}$, then $\hat{T}_{t}$ satisfies \eqref{eq:=order n3} for $0\leq s\leq n+1$. So $\hat{T}_{t}$ is an order $n+1$ deformation which means that $T_{t}$ is extendable.
\end{proof}
\begin{cor}
Let $T$ be an embedding tensor on the $3$-Lie algebra $(\g, [\cdot,\cdot,\cdot]_{\g})$ with respect to the representation $(V ; \rho)$. If $\huaH^{3}(T)=0$, then every $2$-cocycle in $\huaZ^{2}(T)$ is the infinitesimal of some formal deformations of the embedding tensor $T$.
\end{cor}

 \end{document}